\documentclass{amsart}
\usepackage[american]{babel}
\usepackage{amssymb,amsmath,mathrsfs,pgf,tikz}
\usepackage{nicefrac}
\usepackage{hyperref}
\hypersetup{
  colorlinks  = true, 
  urlcolor    = black, 
  linkcolor   = black, 
  citecolor   = black 
}

\newcommand{\NN}{\mathbb{N}}
\newcommand{\ZZ}{\mathbb{Z}}

\newcommand{\G}{\mathscr{G}}

\newcommand{\J}{\mathcal{J}}
\newcommand{\B}{\mathcal{B}}

\newcommand{\FF}{\mathcal{F}}
\newcommand{\PE}{\operatorname{\mathcal{E}}}

\newcommand{\T}{\mathscr{T}}
\newcommand{\E}{\mathscr{E}}
\newcommand{\K}{\mathcal{K}}

\renewcommand{\t}{\mathbf{t}}
\newcommand{\reg}{\operatorname{reg}}
\newcommand{\set}[1]{\left \{ #1 \right \}}

\newcommand{\init}{\operatorname{in}}

\newcommand{\sd}{\triangle}

\newcommand{\ts}{\textstyle}
\renewcommand{\ss}{\scriptstyle}

\newtheorem{theorem}{Theorem}[section]

\newtheorem{lemma}[theorem]{Lemma}

\newtheorem{prop}[theorem]{Proposition}
\newtheorem*{claim*}{Claim}

\theoremstyle{definition}
\newtheorem{definition}[theorem]{Definition}

\newtheorem{example}[theorem]{Example}

\begin{document} 

\title[Eulerian ideals]%
{Eulerian ideals}

\author{J.~Neves}
\address{\emph{J.~Neves}: University of Coimbra, CMUC, Department of Mathematics, 3001-501 Coimbra, Portugal.}
\email{neves@mat.uc.pt}
\thanks{This work was partially supported by the Centre for Mathematics of the University of Coimbra - 
UIDB/00324/2020, funded by the Portuguese Government through FCT/MCTES}

\keywords{Binomial ideal, Castelnuovo--Mumford regularity, $T$-joins}
\subjclass[2010]{13A02, 13H10, 13P25, 05E40; 05C70; 05C45}

\begin{abstract}
Let $G$ be a simple graph and $I(X_G)=\varphi^{-1}(x_i^2-x_j^2 : i,j\in V_G)$, where 
$\varphi \colon K[E_G]\to K[V_G]$ is the homomorphism that sends an edge to the product of its 
vertices. The ideal $I(X_G)$ is Cohen--Macaulay, one-dimensional and binomial. 
If $G$ is bipartite, it is known that the Castelnuovo--Mumford regularity of $I(X_G)$ is equal 
to the maximum cardinality of a set 
of edges having no more than half of the edges of any Eulerian subgraph of $G$.
Here, with respect to the grevlex order associated to an ordering of the edge set of $G$, 
we describe a Gr\"obner basis for $I(X_G)$, and we characterize the standard monomials 
of the ideal $(I(X_G),t_e)$ in terms of even sets of vertices marked with a parity. Using these results,   
we give a combinatorial 
interpretation of the degree of $I(X_G)$,
via the set of even sets of vertices of $G$; and we show that the Castelnuovo--Mumford 
regularity of $I(X_G)$, for any graph, is the maximum cardinality of a set of edges having 
no more than half of the edges of any \emph{even} Eulerian subgraph of $G$ or, equivalently,
the maximum cardinality of a minimum fixed parity $T$-join. 
\end{abstract}
\maketitle

\section{Introduction}
The Eulerian ideal of a graph was introduced by the author, 
Vaz Pinto and Villarreal in \cite{NeVPVi20}. The term Eulerian ideal, which we are introducing here, is owed  
to the relation between a generating set and the set of Eulerian subgraphs with an even cardinality edge set
(even Eulerian subgraphs). 
\smallskip 

Let $G$ be a simple graph without isolated vertices. Denote the set of vertices by $V_G$
and the set of edges by $E_G$. Throughout, we will assume that $E_G$ is a non-empty 
subset of the set of subsets of $V_G$ of cardinality two.
Let $K$ be any field and let 
$$
K[V_G]= K[x_i:i\in V_G],\quad K[E_G] = K[t_h : h\in E_G] 
$$
be the rings of polynomials with coefficients in $K$ whose variables are indexed 
by the vertices and edges of $G$, respectively. Let $\varphi\colon K[E_G]\to K[V_G]$
be given by 
\begin{equation}\label{eq: definition of the map}
\varphi(t_h) = x_ix_j,\;\;  \forall \, h=\set{i,j}\in E_G.
\end{equation}

\begin{definition}\label{def: the ideals} The \emph{Eulerian ideal of $G$} is $I(X_G)=\varphi^{-1}(x_i^2-x_j^2 : i,j\in V_G  )$. 
\end{definition}

The motivation for this definition comes from the notion of va\-ni\-shing ideal over a graph, 
for a \emph{finite} field, introduced by Renter\'ia, Simis and Villarreal in 
\cite{ReSiVi11}. Note that, in the present case, no assumption is made for the field. 
\smallskip 

Let us briefly describe the main features of the Eulerian ideal. 
It is clear from the definition that $I(X_G)$ is a homogeneous ideal. 
Also, $t^2_h-t^2_\ell\in I(X_G)$, for every $h,\ell\in E_G$ and,  
moreover, one can show that any monomial is regular on $K[E_G]/I(X_G)$. 
From this we deduce that $I(X_G)$ has height $|E_G|-1$, 
and that the quotient is a one-dimensional Cohen--Macaulay graded ring. 
Additionally, we know that the ideal is generated by binomials
and these may be associated to even Eulerian subgraphs of the graph (see Definition~\ref{def: Eulerian subgraph}
and Proposition~\ref{prop: key properties}, below).  
\smallskip

These properties were studied in \cite{NeVPVi20} including, also,  
the Castelnuovo--Mumford regularity of $I(X_G)$ in the \emph{bipartite} case. It was shown 
that this invariant is equal to the maximum cardinality of a join, i.e.,  
to the maximum cardinality of a set of edges that has no more than half of the edges 
of any Eulerian subgraph of the graph. This number was termed 
the \emph{maximum vertex join number} by Sol\'e and Zaslavsky and by Frank
(\emph{cf}.~\cite{SoZa93} and \cite{Fr93}). The starting point 
of this work was the extension of this result to the non-bipartite case. 
To achieve this, we define the notion of \emph{parity join}; 
a set of edges that has no more than half of the edges of any \emph{even} cardinality 
Eulerian subgraph of $G$. In Theorem~\ref{thm: regularity and parity joins}
we show that the Castelnuovo--Mumford regularity of 
$I(X_G)$ is the maximum cardinality of a parity join.
\smallskip

To achieve this result, we start by showing that 
a set of homogeneous binomials obtained from the even Eulerian subgraphs, 
together with the set of binomials of the form $t^2_h-t^2_\ell$, for every $h,\ell \in E_G$, form
a Gr\"obner basis for $I(X_G)$ with respect to the graded reverse lexicographic order induced
by a total order of the edges (\emph{cf}.~Theorem~\ref{thm: Grobner basis}). 
The characterization of the outcome of the
division of a monomial by this Gr\"obner basis has lead us to the notion of 
fixed pa\-ri\-ty $T$-joins (\emph{cf.}~Definition~\ref{def: T joins of fixed parity}).
More precisely, by associating a $T$-join to any monomial 
(\emph{cf}.~Definitions~\ref{def: the J operator} and \ref{def: the T set of a monomial}), 
we show that the remainder in a standard expression of the monomial, 
with respect to the aforementioned Gr\"obner basis, yields a $T$-join which has
minimum cardinality among all $T$-joins of same parity cardinality 
(\emph{cf}.~Theorem~\ref{thm: remainder of a monomial by the groebner basis}).
We then describe a bijection between the set of standard monomials of the ideal 
$(I(X_G),t_e)$, with respect to a monomial order as above, and a set of $T$-sets marked with an
element of $\ZZ/2$ (\emph{cf}.~Theorem \ref{thm: standard monomials and T-sets}). 
As the notion of mi\-ni\-mum fixed parity 
$T$-joins and the notion of parity joins are just different ways of describing 
the same set of edges of a graph (cf.~Lemma \ref{lem: parity joins are minimum fixed parity T-joins}), 
the proof of Theorem~\ref{thm: regularity and parity joins}, \emph{i.e.}, the computation 
of the Castelnuovo--Mumford regularity of $I(X_G)$, is then carried
out using fixed parity $T$-joins and the characterization of standard monomials
of $(I(X_G),t_e)$ obtained.
\smallskip

The ideal $I(X_G)$ contains the toric ideal of the graph, $P(G)$, 
which is defined as the kernel of the map 
given by $\eqref{eq: definition of the map}$. These ideals have a longer history and a more intricate nature. 
Their systematic study started with the work of Simis, Vasconcelos and Villarreal \cite{SiVaVi94}.
We know that $P(G)$ is also a binomial ideal and that it is generated by the binomials associated to 
the even closed walks on the graph (\emph{cf}.~\cite[Proposition~3.1]{Vi95}). 
In our case, a set of generators of $I(X_G)$ includes not only these binomials but also any binomial obtained from 
any even Eulerian subgraph of $G$, 
not necessarily connected, and a partition of its edge set into two equal cardinality parts 
(\emph{cf}.~Definitions~\ref{def: Eulerian subgraph} and \ref{def: Groebner basis}).
A contrasting feature to the Eulerian ideal is that, while the former always is, $P(G)$ may rarely be 
Cohen-Macaulay. By way of example, in the recent article \cite{FaKeVT20}, the authors
show that for every pair of integers $d$ and $r$ satisfying $d\geq r\geq 4$, there exists a graph 
yielding a quotient $K[E_G]/P(G)$ with Castelnuovo--Mumford regularity $r$ and $h$-polynomial of degree $d$. 
In recent years, several authors have studied the Castelnuovo--Mumford regularity of the quotient $K[E_G]/P(G)$.
We know that, under mild assumptions, the matching number gives an upper bound for this invariant (\emph{cf}.~\cite{HeHi20}) 
and that lower bounds can be produced from distinguished families of 
induced subgraphs of the graph (\emph{cf}.~\cite{BiO'KVT17,HaBeO'K19}).  
\smallskip

The paper is structured as follows. In the next section we recall the main pro\-perties of $I(X_G)$.
In Section~\ref{sec: gb} we describe a Gr\"obner basis for the Eulerian ideal. 
In Section~\ref{sec: standard mon} we study the combinatorial properties 
of the division of monomials by a Gr\"obner basis and, as a result, 
we give a bijection between the set 
of standard monomials of the ideal $(I(X_G),t_e)$, with respect to 
the graded reverse lexicographic order, in terms of even sets of vertices marked
with a parity. Finally, we apply these results to the computation of the degree and the 
Castelnuovo--Mumford regularity.

\subsection*{Acknowledgments} The author thanks Jens Vygen and Andr\'as Seb\H{o} 
for a helpful discussion on the subject of $T$-joins. The relation between the notions of parity
joins and of fixed parity $T$-joins in Lemma~\ref{lem: parity joins are minimum fixed parity T-joins} 
was pointed out by Andr\'as Seb\H{o}.

\section{Preliminaries}\label{sec: prelim}
Throughout, we will use the multi-index notation to denote monomials in $K[E_G]$. More precisely, 
for each $\alpha\colon E_G \to \mathbb{N}$ of finite support, 
$$
\ts \t^\alpha = \prod\limits_{h \in E_G} t_h^{\alpha(h)}.
$$
We will employ interchangeably the terms \emph{edge} and \emph{variable}. 
In examples, $V_G$ will be a subset of $\NN$ and 
we will abbreviate $t_{\set{i,j}}$ to $t_{ij}$.
\smallskip

For future reference, let us gather in the next proposition 
the main known pro\-per\-ties of the Eulerian ideal of $G$.

\begin{prop}\label{prop: key properties} Let $I(X_G)$ be as in Definition~\ref{def: the ideals} and 
let $\t^\alpha$, $\t^\beta \in K[E_G]$ be relatively prime monomials of the same degree.
\begin{enumerate}
\item $I(X_G)$ is generated by homogeneous binomials.
\item Any monomial is regular on $K[E_G]/I(X_G)$.
\item $K[E_G]/I(X_G)$ is a Cohen--Macaulay, one-dimensional graded ring.  
\item $\t^\alpha-\t^\beta$ belongs to $I(X_G)$ 
if and only if the edge-induced subgraph given by the set of edges raised to odd powers in $\t^\alpha-\t^\beta$ 
has vertices of even degree.
\end{enumerate}
\end{prop}

The proof of (i) uses a standard implicitization argument. One shows that 
$I(X_G)$ is the intersection with $K[E_G]$ of the ideal generated by 
$$
(t_h - x_ix_j z : h=\set{i,j}\in E_G) \cup (x_i^2-x_j^2 : i,j\in V_G)
$$
in the ring $K[E_G,V_G,z]$. 
The proof of (ii) and (iii) are straightforward if the characteristic of the field
is not $2$, in which case, $I(X_G)$ is the vanishing ideal of a set of points in projective space 
with nonzero coordinates (a projective toric subset). For the proof of (i) and (ii) in the general case, the proof 
of (iv) and further details, we refer the reader to \mbox{\cite[Propositions~2.1, 2.2, 2.5 and 2.8]{NeVPVi20}}.

\begin{definition}\label{def: the J operator}
Given $\t^\alpha \in K[E_G]$, let 
$\J(\t^\alpha)=\set{h : \alpha(h)\text{ is odd}}\subset E_G$.
\end{definition}

Using (ii) and (iv) of Proposition~\ref{prop: key properties}, 
we deduce that if $\t^\alpha-\t^\beta$ is a homogeneous binomial then   
$\t^\alpha - \t^\beta \in I(X_G)$ if and only if the edge-induced subgraph of $G$ given by  
symmetric difference $\J(\t^\alpha)\sd \J(\t^\beta)$ has vertices of even degree, i.e., is an Eulerian subgraph. 
Since $\J(\t^\alpha)\sd \J(\t^\beta)\equiv_2 \deg(\t^\alpha) + \deg(\t^\beta)$, 
the Eulerian subgraphs arising from homogeneous binomials of $I(X_G)$ have an edge set of even cardinality.
Let us fix the terminology.

\begin{definition}\label{def: Eulerian subgraph}
A subgraph of $G$ is called \emph{even Eulerian} if its vertices have even degrees and
its edge set has even cardinality.
\end{definition}

\begin{figure}[ht]
\begin{tikzpicture}
\draw[line width=1pt] (4.9,2.34) node[anchor= south] {$\scriptstyle 1$} -- (3.48,0.98);
\draw[line width=1pt] (3.48,0.98) node[anchor= north east] {$\scriptstyle 2$} -- (5.5,0.5);
\draw[line width=1pt] (5.5,0.5) node[anchor= north] {$\scriptstyle 3$} -- (4.9,2.34);
\draw[line width=1pt] (8.02,0.9) node[anchor= north] {$\scriptstyle 4$}-- (9.34,0.56);
\draw[line width=1pt] (9.34,0.56) node[anchor= north] {$\scriptstyle 5$} -- (9.96,1.56);
\draw[line width=1pt] (9.96,1.56) node[anchor= west] {$\scriptstyle 6$} -- (9.2,2.54);
\draw[line width=1pt] (9.2,2.54) node[anchor= south] {$\scriptstyle 7$} -- (8.04,2.16);
\draw[line width=1pt] (8.04,2.16) node[anchor= south] {$\scriptstyle 8$} -- (8.02,0.9);
\draw [color=gray] (8.02,0.9)-- (9.96,1.56);
\draw [color=gray] (5.5,0.5)-- (6.96,0.62);
\draw [color=gray] (6.96,0.62)-- (8.02,0.9);
\draw [color=gray] (8.02,0.9)-- (9.2,2.54);
\draw [color=gray] (3.48,0.98)-- (6.4,1.6);
\draw [color=gray] (6.4,1.6)-- (6.96,0.62);
\draw [color=gray] (3.2,2.46)-- (4.9,2.34);
\draw [color=gray] (4.9,2.34)-- (6.4,1.6);
\draw [color=gray] (6.4,1.6)-- (8.04,2.16);
\begin{scriptsize}
\fill [color=black] (4.9,2.34) circle (2pt);
\fill [color=black] (3.48,0.98) circle (2pt);
\fill [color=black] (5.5,0.5) circle (2pt);
\fill [color=black] (6.96,0.62) circle (2pt);
\fill [color=black] (8.02,0.9) circle (2pt);
\fill [color=black] (9.34,0.56) circle (2pt);
\fill [color=black] (9.96,1.56) circle (2pt);
\fill [color=black] (9.2,2.54) circle (2pt);
\fill [color=black] (8.04,2.16) circle (2pt);
\fill [color=black] (6.4,1.6) circle (2pt);
\fill [color=black] (3.2,2.46) circle (2pt);
\end{scriptsize}
\end{tikzpicture}
\caption{An even Eulerian subgraph of $G$} 
\label{fig: even Eulerian subgraph}
\end{figure}
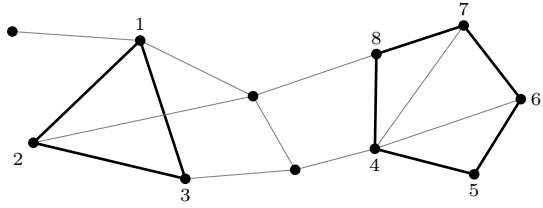

The subgraph in Figure~\ref{fig: even Eulerian subgraph} represented in bold is an even 
Eulerian subgraph. We emphasize that an even Eulerian subgraph is not assumed to 
be a connected graph, or a spanning subgraph.
\smallskip

Let $M$ be a nonzero finitely generated graded module over a 
polynomial ring $S$, and let $F_i = \oplus_{j\in \ZZ} S(-j)^{\beta_{ij}}$ be the free graded modules in 
a minimal graded free resolution of $M$,  
$0\to F_c\to \cdots \to F_1 \to F_0\to M$.
Then the Castelnuovo--Mumford regularity of $M$, referred to in the remainder 
of the text simply by \emph{regularity}, is defined by
$\reg M =\max_{i,j}\set {j-i : \beta_{ij}\not = 0}$.
If $M=0$, we adopt the convention $\reg M = 1$, so that, in particular, when 
$I(X_G) = 0$, $\reg I(X_G) = 1$; such is the case when $G$ consists of a single edge. 
If $M$ is Cohen--Macaulay and $\mathbf{f}\subset S$ is a regular sequence on $M$ of maximum length, 
consisting of elements of degree one, then by \cite[Proposition~4.14]{Ei05} the re\-gu\-la\-ri\-ty of $M$ 
is the largest degree of a nonzero homogeneous element of $M/(\mathbf{f})M$. In our case,
given that $\reg I(X_G) = \reg K[E_G]/I(X_G) +1$, we have the following result.

\begin{prop}\label{prop: regularity reduction}
Let $e\in E_G$ and let $N=K[E_G]/(I(X_G),t_e)$. Then 
$$
\reg I(X_G) = \max \set{d : N_d\not = 0} +1.
$$
\end{prop}

In \cite[Proposition~3.2]{NeVPVi20}, the regularity of $I(X_G)$ for some special families
of graphs is given. We list them in Table~\ref{tbl: values for special families of graphs}. 
\begin{table}[ht]
\renewcommand{\arraystretch}{1.3}
\begin{tabular}{l|l}
$G$ & $\reg I(X_G)$\\
\hline 
Forest & $|E_G|$ \\
Even cycle & $|E_G|/2$ \\
Complete bipartite $\K_{a,b}$ & $\max\set{a,b}$\\
Non-bipartite, uni-cyclic & $|E_G|$ \\
Complete graph, $\K_n,\, n\geq 4$ & $\lfloor\frac{n}{2} \rfloor+1$ \\
\end{tabular}
\bigskip 

\caption{$\reg I(X_G)$ for special families of graphs.} 
\label{tbl: values for special families of graphs}
\end{table}

\section{A Gr\"obner basis}\label{sec: gb}
Let $G$ be as in Figure~\ref{fig: even Eulerian subgraph} and, associated to 
the even Eulerian subgraph in bold, consider the monomials
$\t^\alpha = t_{12}t_{23}t_{13}t_{48}$ and $\t^\beta = t_{45}t_{56}t_{67}t_{78}$.
Let $\varphi$ be the map given by \eqref{eq: definition of the map}. Then,
using Definition~\ref{def: the ideals},
$$
\varphi (\t^\alpha-\t^\beta) = (x_1^2x_2^2x_3^2 - x_5^2x_6^2x_7^2)x_4x_8\implies \t^\alpha-\t^\beta \in I(X_G).
$$
By (iv) of Proposition~\ref{prop: key properties}, the conclusion is the same if we take 
any other even Eulerian subgraph of $G$ and any other partition 
of its edge set in two parts of equal cardinality. This motivates the following definition.

\begin{definition}\label{def: Groebner basis}
A binomial $\t^\alpha-\t^\beta$ is called 
\emph{Eulerian} if $\t^\alpha$ and $\t^\beta$ are distinct, relatively prime, square-free monomials
of same degree and the edge-induced subgraph given by $\J(\t^\alpha)\cup \J(\t^\beta)$ 
is an even Eulerian subgraph of $G$. Denoting the set of Eulerian binomials by $\E$
and $\T = \set{t_h^2 - t_\ell^2 : h,\ell \in E_G}$, define $\G= \T \cup \E$. 
\end{definition}

Of course, not all homogeneous binomials in $I(X_G)$ are Eulerian;
it suffices that one of $\t^\alpha$ or $\t^\beta$ is not square-free. Clearly, 
$\G$ is a finite set. We will show that it is a Gr\"obner 
basis of $I(X_G)$ for any graded reverse lexicographic order on $K[E_G]$, in the following sense.

\begin{definition}\label{def: grevlex order}
By a grevlex order on $K[E_G]$ we will mean the graded reverse lexicographic 
order on $K[E_G]$ induced by $t_{\epsilon(1)}\succ t_{\epsilon(2)} \succ \cdots \succ t_{\epsilon(s)}$, 
where $\epsilon$ is a bijection 
$\epsilon\colon \set{1,\dots,s}\to E_G$. If $e=\epsilon_{s}$ is the last edge, then 
such a monomial order will be referred to as a grevlex order on $K[E_G]$ with $t_e$ last.
\end{definition}

Recall that if $\t^\alpha$ and $\t^\beta$ are two 
monomials of the same degree and $h$ is the last edge in the support of $\beta-\alpha$, 
then $\t^\alpha \succ \t^\beta$ in the grevlex order if and only if $\beta(h)>\alpha(h)$. 
In particular, if $\t^\alpha$ and $\t^\beta$ are also square-free, then 
$\init_\prec(\t^\alpha-\t^\beta) = \t^\alpha$ if and only if $\J(\t^\beta)$ contains the 
last edge of $\J(\t^\alpha)\sd \J(\t^\beta)$.

\begin{theorem}\label{thm: Grobner basis}
Assume $|E_G|>1$ and fix $\prec$ a grevlex order on $K[E_G]$. Then $\G$ is a Gr\"obner basis of $I(X_G)$ with respect to $\prec$.
\end{theorem}

\begin{proof}
Let us begin by showing that $I(X_G)=(\G)$. In view of Proposition~\ref{prop: key properties},
the inclusion $(\G)\subset I(X_G)$ is clear. Since $I(X_G)$ is generated by binomials, to 
prove the opposite inclusion it suffices to show that any 
homogeneous binomial $\t^\alpha - \t^\beta$ in $I(X_G)$ belongs to $(\G)$. 
By (ii) of Proposition~\ref{prop: key properties} we may assume $\gcd(\t^\alpha,\t^\beta) = 1$.
We will use induction on the degree of the binomial. 
By (iv) of Proposition~\ref{prop: key properties}, there are no homogeneous binomials in $I(X_G)$
of degree one, so the base case when degree is equal to two. In this case, 
either both $\t^\alpha$ and $\t^\beta$ are squares of variables, 
or neither of them is, in which case $\t^\alpha-\t^\beta$ is an Eulerian binomial. 
Either way, we get $\t^\alpha-\t^\beta\in (\G)$.  
Suppose now that $\t^\alpha-\t^\beta \in I(X_G)$ has degree larger than or equal to three.
If $\t^\alpha$ and $\t^\beta$ are also square-free then $\t^\alpha-\t^\beta\in \E$.
Suppose that this is not the case, suppose that, say, $\t^\alpha$ is divisible by $t_h^2$, for some $h\in E_G$.
Then, choose $\ell\in E_G$ such that $t_\ell$ divides $\t^\beta$ and write $\t^\alpha = t_h^2 \t^\gamma$ and 
$\t^\beta = t_\ell \t^\mu$. From 
$$
\t^\alpha-\t^\beta = t_\ell(t_\ell \t^\gamma - \t^\mu)  + (t_h^2 - t_\ell^2) \t^\gamma,
$$
and the fact that $(t_h^2 - t_\ell^2) \t^\gamma \in (\G)\subset I(X_G)$, we deduce that      
$$
t_\ell(t_\ell \t^\gamma - \t^\mu) \in I(X_G) \implies t_\ell \t^\gamma - \t^\mu \in I(X_G).
$$
By induction, we get $t_\ell \t^\gamma - \t^\mu \in (\G)$ and thus $\t^\alpha-\t^\beta \in (\G)$. 
This concludes the proof that $I(X_G) = J=(\G)$.
\smallskip

To prove that $\G$ is a Gr\"obner basis of $I(X_G)$ with respect to $\prec$, 
let us show that $S(f,g)$ reduces to zero with respect to $\G$, for every $f,g\in \G$.
We only need to consider $f$ and $g$ for which $\gcd(\init_\prec(f),\init_\prec(g))\not = 1$.
There are three cases. 
\smallskip

\noindent
If $f,g\in \T$ and $\gcd(\init_\prec(f),\init_\prec(g))\not = 1$, then, without loss of generality, we may assume 
$f$ and $g$ are of the form $f = t_h^2 - t^2_\ell$ and $g=t_h^2 - t^2_e$ with $\init_\prec(f) = \init_\prec(g) = t^2_h$.
Then $S(f,g)= f-g = t_e^2-t^2_\ell \in \T$.
\smallskip

\noindent
If $f\in \T$, $g\in \E$ and $\gcd(\init_\prec(f),\init_\prec(g))\not = 1$, 
then, without loss of generality, we may assume that \mbox{$f=t^2_h-t^2_e$} and 
$g=t_h\t^\gamma - t_\ell\t^\mu$ with $\init_\prec(f)=t^2_h$, $\init_\prec(g) = t_h\t^\gamma$ and 
$\ell$ the last variable of $\J(t_h\t^\gamma)\cup\J(t_\ell\t^\mu)$. Then, since 
$t_h\t^\mu - t_\ell \t^\gamma \in \E$,  
and $\init_\prec (t_h\t^\mu - t_\ell \t^\gamma) = t_h\t^\mu$,
$$
S(f,g) =  t_\ell t_h\t^\mu -t_e^2 \t^\gamma\stackrel{\E}{\longrightarrow}
(t_\ell^2 - t_e^2)\t^\gamma \stackrel{\T}{\longrightarrow} 0.
$$

\noindent
Finally, suppose that $f,g\in \E$ and $\gcd(\init_\prec(f),\init_\prec(g))\not = 1$. Let $f = \t^{\delta+\gamma} - \t^\epsilon$ and 
$g = \t^{\delta+\mu} - \t^\nu$, with \mbox{$\init_\prec(f)=\t^{\delta + \gamma}$}, $\init_\prec(g) = \t^{\delta + \mu}$ 
and $\gcd(\t^\gamma,\t^\mu)=1$. Then 
\begin{equation}\label{eq: L503}
S(f,g) = \t^{\gamma+\nu} - \t^{\mu+\epsilon}.
\end{equation}
Since $f,g\in \E$, each of the two sets $\J(\t^\delta\t^\gamma)\cup\J(\t^\epsilon)$ and 
$\J(\t^\delta\t^\mu)\cup \J(\t^\nu)$ defines an even Eulerian subgraph of $G$. It follows that 
the symmetric difference of the two sets also defines an even Eulerian subgraph of $G$. As
$$
\renewcommand{\arraystretch}{1.3}
\begin{array}{l}
\phantom{=}(\J(\t^\delta\t^\gamma)\cup\J(\t^\epsilon))\sd (\J(\t^\delta\t^\mu)\cup \J(\t^\nu))
= \J(\t^{\gamma+\nu})\sd \J(\t^{\mu+\epsilon}),
\end{array}
$$
this means that $\J(\t^{\gamma+\nu})\sd \J(\t^{\mu+\epsilon})$
defines an even Eulerian subgraph of $G$. If  
$\t^{\gamma+\nu}$ and $\t^{\mu+\epsilon}$ are relatively 
prime and square-free we get $S(f,g)\in \E$ which, trivially, 
reduces to zero with respect to $\G$. 
Suppose this is not the case. Consider the monomials
$\t^\zeta = \gcd(\t^{\gamma+\nu} , \t^{\epsilon +\mu})$, $ \t^\phi = \gcd(\t^{\gamma} , \t^{\nu})$ and 
$\t^\psi = \gcd (\t^{\mu} ,\t^{\epsilon} )$. As, by assumption, 
$\gcd(\t^{\gamma},\t^\epsilon) = \gcd(\t^\mu,\t^\nu)=1$, we deduce that 
$\gcd(\t^\zeta,\t^\phi)=\gcd(\t^\zeta,\t^\psi)=1$. Hence, 
there exist $\t^\alpha,\t^\beta \in K[E_G]$, relatively prime, square-free 
monomials, such that   
\begin{equation}\label{eq: L501}
\t^{\gamma+\nu-\zeta} = \t^\alpha (\t^\phi)^2\quad \text{and} \quad \t^{\epsilon+\mu-\zeta} = \t^\beta (\t^\psi)^2.
\end{equation}
From the fact that $\J(\t^{\gamma+\nu})\sd \J(\t^{\mu+\epsilon})$ defines an even Eulerian subgraph 
we deduce that $\J(\t^\alpha)\sd \J(\t^\beta)$ also defines an even Eulerian subgraph. 
Let $a = \deg(\t^\phi)$ and $b=\deg (\t^\psi)$. Then, letting $e$ denote the last variable of $E_G$, 
\begin{equation}\label{eq: L516}
\renewcommand{\arraystretch}{1.3}
\begin{array}{l}
\deg(\t^\alpha)+2a = \deg S(f,g) - \deg(\t^\zeta) = \deg(\t^\beta) + 2b, \\
S(f,g) = \t^\zeta(\t^\alpha (\t^\phi)^2 - \t^\beta (\t^\psi)^2) \stackrel{\T}{\longrightarrow} \t^\zeta (\t^\alpha t_e^{2a} - \t^\beta t_e^{2b}).
\end{array}
\end{equation}
If $a=b$ then $\deg(\t^\alpha)=\deg(\t^\beta)$ and thus $\t^\alpha-\t^\beta\in\E$. 
From \eqref{eq: L516} we deduce that $S(f,g)$ reduces to zero with respect to $\E$.
Consider now the case $a\not = b$, and assume, without loss of generality, that $a<b$ and therefore 
$\deg(\t^\alpha)>\deg(\t^\beta)$.  
By Lemma~\ref{lem: two binomials nearly Eulerian} (proved below) there exists $\t^\xi\in K[E_G]$ such that 
$\t^{\alpha-\xi} - \t^{\beta+\xi}\in \E$  and $\init_\prec(\t^{\alpha-\xi} - \t^{\beta+\xi}) = \t^{\alpha-\xi}$.
It follows that $2\deg(\t^\xi) = \deg(\t^\alpha)-\deg(\t^\beta)=2b-2a$.
Continuing from \eqref{eq: L516}, 
$$
S(f,g) \stackrel{\G}{\longrightarrow} 
\t^\zeta \t^{\beta}((\t^\xi)^2 t_e^{2a}  -  t_e^{2b} )
\stackrel{\T}{\longrightarrow} 0.
$$
This finishes the proof of the theorem.
\end{proof}

\begin{lemma}\label{lem: two binomials nearly Eulerian}
Let $\prec$ be a grevlex order on $K[E_G]$. 
Suppose that  $\t^\alpha,\t^\beta\in K[E_G]$ are relatively prime, square-free monomials, 
with $\deg(\t^\alpha)>\deg(\t^\beta)$ and such that $\J(\t^\alpha)\cup\J(\t^\beta)$
defines an even Eulerian subgraph. Then there exists $\t^\xi\in K[E_G]$
dividing $\t^\alpha$ such that $\t^{\alpha-\xi}-\t^{\beta+\xi}\in \E$ and 
$\init_ \prec (\t^{\alpha-\xi}-\t^{\beta+\xi}) = \t^{\alpha-\xi}$.
In particular, $\t^\alpha$ is divisible by a leading term of an element of $\E$.
\end{lemma}

\begin{proof}
Set $2d=\deg(\t^\alpha)-\deg(\t^\beta)$ and let $\t^\xi$ be the product of the $d$ last variables in $\J(\t^\alpha)$. Then 
$\t^{\alpha-\xi}$ and $\t^{\beta+\xi}$ are relatively prime, square-free monomials of equal degree 
such that $\J(\t^{\alpha-\xi})\cup\J(\t^{\beta+\xi}) = \J(\t^\alpha)\cup \J(\t^\beta)$ defines an even Eulerian 
subgraph. We deduce that $\t^{\alpha-\xi}-\t^{\beta+\xi}\in \E$. As the last edge in 
$\J(\t^{\alpha-\xi})\cup\J(\t^{\beta+\xi})$ belongs to $\J(\t^{\beta+\xi})$ we deduce
that $\init_ \prec (\t^{\alpha-\xi}-\t^{\beta+\xi}) = \t^{\alpha-\xi}$.
\end{proof}

\section{Standard Monomials}\label{sec: standard mon}

Let $T\subset V_G$ be a (possibly empty) set of vertices. A  $T$-join is a subset of edges $J\subset E_G$ such that 
$T$ is precisely the set of odd degree vertices of the subgraph of $G$ edge-induced by 
$J$. Since the number of odd degree vertices of a graph is even, the existence of a 
$T$-join implies that the intersection of $T$ with the vertex set of every 
connected component of the graph has even cardinality.

\begin{definition}
A subset $T\subset V_G$ is called an \emph{even subset of vertices} 
if $|T\cap V_H|$ is even, for every connected component $H\subset G$. 
We denote the set of all even subsets of vertices of a graph $G$ by $\PE(V_G)$.
\end{definition}

\begin{figure}[ht]
\begin{tikzpicture}

\coordinate (P1) at (-1,-.3);
\coordinate (P2) at (-.6,1.25);
\coordinate (P3) at (-2.25,.75);
\coordinate (P4) at (.5,-.5);
\coordinate (P5) at (-2,-.1);
\coordinate (P6) at (.25,.25);

\coordinate (P7) at (1.5,-.25);
\coordinate (P8) at (2.15,.7);
\coordinate (P9) at (3,1.2);
\coordinate (P10) at (4.8,.7);
\coordinate (P11) at (2.5,-.5);
\coordinate (P12) at (4.25,-.25);

\draw (P1)--(P2);
\draw (P1)--(P3);
\draw[line width=1pt] (P1)--(P5);
\draw[line width=1pt] (P1)--(P6);
\draw[line width=1pt] (P3)--(P2);
\draw[line width=1pt] (P6)--(P2);

\draw[line width=1pt] (P4)--(P7);

\draw (P9)--(P11);
\draw (P9)--(P12);
\draw[white,fill=white] (2.85,.7) circle (1.5pt);
\draw[white,fill=white] (2.77,.42) circle (1.5pt);
\draw[white,fill=white] (3.42,.7) circle (1.5pt);
\draw[line width=1pt] (P8)--(P12);
\draw[line width=1pt] (P8)--(P10);
\draw (P9)--(P10);
\draw[line width=1pt] (P8)--(P11);
\draw (P10)--(P12);

\foreach \i in {1,...,12} {\fill[color=black] (P\i) circle (2pt);}

\foreach \i in {1} {\draw  (P\i) node[anchor=north] {$\ss \i$};}
\foreach \i in {2,4,7,8,9} {\draw  (P\i) node[anchor=south] {$\ss \i$};}
\foreach \i in {3,5} {\draw  (P\i) node[anchor=east] {$\ss \i$};}
\foreach \i in {6,10,11,12} {\draw  (P\i) node[anchor=west] {$\ss \i$};}

\end{tikzpicture}
\caption{$T$-joins.} 
\label{fig: T-joins}
\end{figure}

When $T$ consists of two vertices in the same connected component, a $T$-join is an 
edge-disjoint union of a path between the two vertices together with cycles. 
A minimal cardinality $T$-join is then a shortest path between the two vertices. In 
the graph of Figure~\ref{fig: T-joins}, the path in bold on the left is a $\set{3,5}$-join,
but not a minimal cardinality one. For a non-empty $T\in \PE(V_G)$, a $T$-join may constructed 
by choosing paths between pairs of vertices of $T$ in a same connected component and taking 
their symmetric difference. If $T=\emptyset$, then the empty set (of edges) is a $T$-join. 
Any non-empty $T$-join, in this case, is a subset of edges defining a subgraph of $G$ 
with vertices of even degree or, according to Definition~\ref{def: Eulerian subgraph}, 
an Eulerian subgraph of $G$.
We see that a $T$-join always exists. Whether, for a given $T\in \PE(V_G)$, 
a $T$-join with cardinality of a given parity exists 
is a different matter, which we will explore. We refer the reader to \cite[Chapter~12]{KoVy06} 
and \cite[Chapter~29]{Sch03} for further properties of  
$T$-joins.

\begin{definition}\label{def: the T set of a monomial}
Given $\t^\alpha \in K[E_G]$, recalling from Definition~\ref{def: the J operator}, that 
$\J(\t^\alpha)$ is the subset of edges $\set{h\in E_G : \alpha(h)\text{ is odd}}$, 
define $\theta(\t^\alpha)\in \PE(V_G)$ 
to be the set of odd degree vertices of the subgraph edge-induced by $\J(\t^\alpha)$. 
\end{definition}

Note that, by definition, $\J(\t^\alpha)$ is a $\theta(\t^\alpha)$-join. Additionally, 
if $\t^\alpha,\t^\beta\in K[E_G]$,
as $\J(\t^\alpha\t^\beta) = \J(\t^\alpha)\sd \J(\t^\beta)$, we deduce that
$\J(\t^\alpha \t^\beta)$ is $(\theta(\t^\alpha)\sd \theta(\t^\beta))$-join. This follows from 
an elementary property of $T$-joins (\emph{cf}.~\cite[Proposition~12.6]{KoVy06}). This implies 
that $\theta(\t^\alpha\t^\beta)= \theta(\t^\alpha)\sd \theta(\t^\beta)$. In particular,
$\J(\t^\alpha)\sd \J(\t^\beta)$ is an Eulerian subgraph if and only if
$$\theta(\t^\alpha)\sd \theta(\t^\beta) = \emptyset \iff \theta(\t^\alpha)= \theta(\t^\beta).$$
We will use this property in the sequel.

\smallskip

As we shall see, $T$-joins play an important role in the characterization of monomials.
We start by showing that they can be used to characterize the process of division of monomials 
by the Gr\"obner basis that was introduced in Definition~\ref{def: Groebner basis}.

\begin{theorem}\label{thm: remainder of a monomial by the groebner basis} 
Assume $|E_G|>1$, fix a grevlex order on $K[E_G]$ and let $\G$ be the associated Gr\"obner basis of $I(X_G)$. 
Let $\t^\delta,\t^\gamma \in K[E_G]$.
\begin{enumerate}
\item If the monomial $\t^\delta$ is the remainder of the division of $\t^\gamma$ by an element of 
$\G$ then \mbox{$\theta(\t^\delta)=\theta(\t^\gamma)$} and $|\J(\t^\delta)|\equiv_2 |\J(\t^\gamma)|$.
\item The remainder in a standard expression of 
$\t^\gamma$ with respect to $\G$ is a monomial, $\t^\delta$, with
$|\J(\t^\delta)| = \min\set{|J| : J\text{ is a $\theta(\t^\gamma)$-join and $|J|\equiv_2|\J(\t^\gamma)|$}}$.
\end{enumerate}
\end{theorem}

\begin{proof}
Let $e\in E_G$ be the last edge. As $\G=\T\cup \E$, there are two cases in the proof of (i).
If $\t^\gamma$ is divisible by $t^2_h$, for some $h\not = e$, 
the division of $\t^\gamma$ by the element $t_h^2-t_\ell^2$ yields remainder 
$\t^\delta = \t^\gamma t_h^{-2}t_\ell^2$. In this case, $\J(\t^\delta)=\J(\t^\gamma)$ and (i) follows trivially. 
Suppose now that $\t^\gamma$ is divisible by $\t^\alpha$ where $\t^\alpha-\t^\beta$ is an Eulerian 
binomial. Let 
$\t^\gamma=\t^\rho \t^\alpha$, for some $\t^\rho \in K[E_G]$. Then, 
division yields remainder $\t^\delta =  \t^\rho \t^\beta$.
Let $T_1=\theta (\t^\alpha)$, $T_2=\theta(\t^\beta)$ and $T_3=\theta(\t^\rho)$. 
Since $\t^\alpha-\t^\beta$ is Eulerian, $T_1=T_2$. From 
$$
\renewcommand{\arraystretch}{1.4}
\begin{array}{l}
\J(\t^\gamma) = \J(\t^\rho)\sd \J(\t^\alpha)\equiv_2 |\J(\t^\rho)|+|\J(\t^\alpha)|\\
\theta(\t^\gamma) = \theta(\t^\rho)\sd \theta(\t^\alpha) = T_3\sd T_1 = T_3\sd T_2,
\end{array}
$$
we deduce that $\J(\t^\delta)=\J(\t^\rho)\sd \J(\t^\beta)$ is a $\theta(\t^\gamma)$-join and, 
moreover,
$$
\renewcommand{\arraystretch}{1.3}
\begin{array}{l}
|\J(\t^\delta)| = |\J(\t^\rho)| + |\J(\t^\beta)| - 2 |\J(\t^\rho)\cap \J(\t^\beta)| \\
\phantom{|\J(\t^\delta)|} \equiv_2 |\J(\t^\rho)| + |\J(\t^\alpha)| \equiv_2 |\J(\t^\gamma)|.
\end{array}
$$

To prove (ii) we start by remarking that, since $\G$ consists of binomials, the remainder term 
in a standard expression of a monomial with respect to $\G$ is also a monomial. 
Let us denote 
$$
\Gamma=\set{J : J\text{ is a $\theta(\t^\gamma)$-join and $|J|\equiv_2|\J(\t^\gamma)|$}}.
$$
Since $\Gamma$ is non-empty ($\J(\t^\gamma)$ belongs to it) we may consider $\t^\rho$, 
the product of the edges of a minimum cardinality element of $\Gamma$. 
Let $\t^\nu$ be the remainder in a standard expression of $\t^\rho$ 
with respect to $\G$. By (i), $J(\t^\nu)\in \Gamma$.
Since 
$$
|\J(\t^\nu)|\leq \deg(\t^\nu) = \deg(\t^\rho),
$$
by the minimality, we deduce that $\t^\nu$ is square-free.
On the other hand, let $\t^\mu$ be a square-free monomial such that 
$\t^\delta = \t^\mu t_e^{2k}$, for some $k\geq 0$. 
Then $\J(\t^\delta) = \J(\t^\mu)$. Using (i), 
$$
\theta(\t^\gamma)=\theta(\t^\delta)=\theta(\t^\mu).
$$
It suffices to prove that $\t^\mu = \t^\nu$. Arguing by contradiction,  
assume that $\t^\mu \not =\t^\nu$ and let $\t^\alpha = \t^\mu \gcd(\t^\mu,\t^\nu)^{-1}$ 
and $\t^\beta = \t^\nu \gcd(\t^\mu,\t^\nu)^{-1}$. Then, $\t^\alpha$ and $\t^\beta$ are distinct, 
relatively prime, square-free monomials, satisfying $\theta(\t^\alpha) = \theta(\t^\beta)$. 
If they have equal degree then $\t^\alpha-\t^\beta$ is Eulerian and hence 
one of $\t^\mu$ or $\t^\nu$ is divisible by a leading term of $\G$, which is absurd. 
If $\deg(\t^\alpha)\not = \deg(\t^\beta)$ then, Lemma~\ref{lem: two binomials nearly Eulerian}
yields the same conclusion. 
\end{proof}

Minimum $T$-joins, i.e., $T$-joins with minimum cardinality are of special 
interest in questions of Combinatorial Optimization. Here, Theorem~\ref{thm: remainder of a monomial by the groebner basis}
is leading us to a refinement of this notion, namely, minimum cardinality $T$-joins among 
$T$-joins of a fixed parity cardinality. As we show below, if $G$ is bipartite the parity of the cardinality 
of $T$-joins (for a fixed $T$-join) does not change. However if $G$ is non-bipartite, the sets of $T$-joins of even cardinality and 
and of odd cardinality are both non-empty.

\begin{definition}\label{def: T joins of fixed parity}
Let $T\in\PE(V_G)$, $J$ a $T$-join and set $i=|J|+2\ZZ \in \ZZ/2$. 
We will say that $J$ is a $T$-join of parity $i$. Let us denote by $\J_i(G,T)$ the 
set of all $T$-joins of parity $i$ and let  
$\tau_i(G,T)$ denote the minimum cardinality of an element of $\J_i(G,T)$. 
\end{definition}

To ease notation, we will denote the elements of $\ZZ/2$ by $0$ and $1$. 
Naturally, \emph{parity zero} and \emph{even} will be used as synonyms, and the same 
applies to \emph{parity one} and \emph{odd}. Note that $\tau_i(G,T)$  
is defined only if $\J_i(G,T)$ is non-empty. We will refer to the minimum cardinality elements of  
$\J_i(G,T)$ as \emph{minimum fixed parity} $T$-joins. Minimum fixed parity $T$-joins also appear 
in Combinatorial Optimization; they are solutions of the a so-called \emph{Parity Join Problem} (\emph{cf}.~\cite{GeeKa18}).

\begin{example} 
Let $G$ be the graph in Figure~\ref{fig: T-joins}. The path depicted in bold, on the left, 
is an even $\set{3,5}$-join. 
It is not a minimum even $\set{3,5}$-join as $3$ and $5$ are 
joined by a path of length two. We deduce that $\tau_0(G,\set{3,5})=2$. 
Other examples of computations of these numbers are: 
\mbox{$\tau_1(G,\set{3,5})=3$}, \mbox{$\tau_0(G,\set{4,7})=4$} 
(take the edge between $4$ and $7$ and any triangle), $\tau_1(G,\set{4,7})=1$, 
$\tau_0(G,\set{8,10,11,12})=2$ and $\tau_1(G,\set{8,10,11,12})=3$. 
\end{example}

For a fixed $T\in \PE(V_G)$, the minimum cardinality of $T$-joins 
(without the parity constraint) is denoted in the literature by $\tau(G,T)$. 
If $G$ admits $T$-joins of both parities
then, of course, $\tau(G,T) = \min\set{\tau_0(G,T),\tau_1(G,T)}$.

\begin{lemma}\label{lem: when do odd and even T-joins exist}
Let $T\in \PE(V_G)$. Then, $G$ is non-bipartite if and only if $\J_0(G,T)$ and $\J_1(G,T)$ are both non-empty.
\end{lemma}

\begin{proof}
If there exist $T$-joins, $J_1$ and $J_2$, such that $|J_1|\not \equiv_2 |J_2|$, 
then $J_1\sd J_2$ is non-empty and 
defines an Eulerian subgraph $C\subset G$ with 
$$
|E_C|=|J_1\sd J_2| \equiv_2 |J_1|+|J_2| \equiv_2 1. 
$$
As $C$ decomposes into an edge-disjoint union of cycles, one of these cycles must be odd and thus  
$G$ is non-bipartite. 
Conversely if $G$ is non-bipartite and $C\subset G$ is an odd cycle then, 
for a $T$-join, $J\subset E_G$, the subset $J\sd E_C$ 
is a $T$-join with $|J\sd E_C|\not \equiv_2 |J|$.
\end{proof}

\begin{example}
Let $G=\K_n$ be a complete graph on $n\geq 3$ vertices. 
Table~\ref{tbl: minimum and odd T-joins in Kn} contains a 
list of minimum cardinality elements of $\J_0(G,T)$ and $\J_1(G,T)$, for $T$ of varying cardinality. 
\begin{table}[ht]
\renewcommand{\arraystretch}{1.3}
\begin{tabular}{c|c|c}
$|T|$ & parity $0$ & parity $1$\\
\hline 
$0$ & $\emptyset$ & \begin{tikzpicture}[scale=1.25]
\draw [color=black] (0,0)-- (.25,0);
\draw [color=black] (.5,0)-- (.25,0);
\draw [color=black] (0,0)..controls (0.1,.15) and (.4,.15).. (.5,0);
\fill [color=gray!45!white] (0,0) circle (1.25pt);
\fill [color=gray!45!white] (.25,0) circle (1.25pt);
\fill [color=gray!45!white] (.5,0) circle (1.25pt);
\end{tikzpicture}\\
$2$ & \begin{tikzpicture}[scale=1.25]
\draw [color=black] (0,0)-- (.25,0);
\draw [color=black] (.5,0)-- (.25,0);
\fill [color=black] (0,0) circle (1.25pt);
\fill [color=gray!45!white] (.25,0) circle (1.25pt);
\fill [color=black] (.5,0) circle (1.25pt);
\end{tikzpicture} & \begin{tikzpicture}[scale=1.25]
\draw [color=black] (0,0)-- (.25,0);
\fill [color=black] (0,0) circle (1.25pt);
\fill [color=black] (.25,0) circle (1.25pt);
\end{tikzpicture}\\
$4$ & \begin{tikzpicture}[scale=1.25]
\draw [color=black] (0,0)-- (.25,0);
\draw [color=black] (0.5,0)-- (.75,0);
\fill [color=black] (0,0) circle (1.25pt);
\fill [color=black] (.25,0) circle (1.25pt);
\fill [color=black] (0.5,0) circle (1.25pt);
\fill [color=black] (.75,0) circle (1.25pt);
\end{tikzpicture} & 
\begin{tikzpicture}[scale=1.25]
\draw [color=black] (0.5,0)-- (.75,0);
\draw [color=black] (0,0)..controls (0.1,.2) and (.65,.2).. (.75,0);
\draw [color=black] (.25,0)..controls (0.35,.1) and (.65,.1).. (.75,0);
\fill [color=black] (0,0) circle (1.25pt);
\fill [color=black] (.25,0) circle (1.25pt);
\fill [color=black] (0.5,0) circle (1.25pt);
\fill [color=black] (.75,0) circle (1.25pt);
\end{tikzpicture}\\
$6$ & \begin{tikzpicture}[scale=1.25]
\draw [color=black] (-0.5,0)-- (-.25,0);
\draw [color=black] (0.5,0)-- (.75,0);
\draw [color=black] (0,0)..controls (0.1,.2) and (.65,.2).. (.75,0);
\draw [color=black] (.25,0)..controls (0.35,.1) and (.65,.1).. (.75,0);
\fill [color=black] (0,0) circle (1.25pt);
\fill [color=black] (.25,0) circle (1.25pt);
\fill [color=black] (0.5,0) circle (1.25pt);
\fill [color=black] (.75,0) circle (1.25pt);
\fill [color=black] (-.5,0) circle (1.25pt);
\fill [color=black] (-.25,0) circle (1.25pt);
\end{tikzpicture} & \begin{tikzpicture}[scale=1.25]
\draw [color=black] (0,0)-- (.25,0);
\draw [color=black] (0.5,0)-- (.75,0);
\draw [color=black] (1,0)-- (1.25,0);
\fill [color=black] (0,0) circle (1.25pt);
\fill [color=black] (.25,0) circle (1.25pt);
\fill [color=black] (0.5,0) circle (1.25pt);
\fill [color=black] (.75,0) circle (1.25pt);
\fill [color=black] (1,0) circle (1.25pt);
\fill [color=black] (1.25,0) circle (1.25pt);
\end{tikzpicture}\\
$8$ & \begin{tikzpicture}[scale=1.25]
\draw [color=black] (0,0)-- (.25,0);
\draw [color=black] (0.5,0)-- (.75,0);
\draw [color=black] (1,0)-- (1.25,0);
\draw [color=black] (1.5,0)-- (1.75,0);
\fill [color=black] (0,0) circle (1.25pt);
\fill [color=black] (.25,0) circle (1.25pt);
\fill [color=black] (0.5,0) circle (1.25pt);
\fill [color=black] (.75,0) circle (1.25pt);
\fill [color=black] (1,0) circle (1.25pt);
\fill [color=black] (1.25,0) circle (1.25pt);
\fill [color=black] (1.5,0) circle (1.25pt);
\fill [color=black] (1.75,0) circle (1.25pt);
\end{tikzpicture} &
\begin{tikzpicture}[scale=1.25]
\draw [color=black] (-1,0)-- (-.75,0);
\draw [color=black] (-0.5,0)-- (-.25,0);
\draw [color=black] (0.5,0)-- (.75,0);
\draw [color=black] (0,0)..controls (0.1,.2) and (.65,.2).. (.75,0);
\draw [color=black] (.25,0)..controls (0.35,.1) and (.65,.1).. (.75,0);
\fill [color=black] (0,0) circle (1.25pt);
\fill [color=black] (.25,0) circle (1.25pt);
\fill [color=black] (0.5,0) circle (1.25pt);
\fill [color=black] (.75,0) circle (1.25pt);
\fill [color=black] (-.5,0) circle (1.25pt);
\fill [color=black] (-.25,0) circle (1.25pt);
\fill [color=black] (-1,0) circle (1.25pt);
\fill [color=black] (-.75,0) circle (1.25pt);
\end{tikzpicture}\\
\begin{tikzpicture}[scale=1]
\fill [color=white] (0,0.09) circle (.05pt);
\fill [color=black] (0,0.1) circle (.5pt);
\fill [color=black] (0,0.2) circle (.5pt);
\fill [color=black] (0,0.3) circle (.5pt);
\end{tikzpicture} &
\begin{tikzpicture}[scale=1]
\fill [color=black] (0,0.1) circle (.5pt);
\fill [color=black] (0,0.2) circle (.5pt);
\fill [color=black] (0,0.3) circle (.5pt);
\end{tikzpicture} &
 \begin{tikzpicture}[scale=1]
\fill [color=black] (0,0.1) circle (.5pt);
\fill [color=black] (0,0.2) circle (.5pt);
\fill [color=black] (0,0.3) circle (.5pt);
\end{tikzpicture}
\end{tabular}
\bigskip 

\caption{Minimum even and odd $T$-joins in $\K_n$} 
\label{tbl: minimum and odd T-joins in Kn}
\end{table}
In this table the vertices in $T$ are represented in black, while for other vertices 
(in the cases of $|T|=0$ and $1$) we use gray. 
Since a $T$-join has at least $|T|/2$ edges it is easy to check 
that the $T$-joins listed in the two columns of Table~\ref{tbl: minimum and odd T-joins in Kn} 
are minimum cardinality elements of $\J_0(G,T)$ and $\J_1(G,T)$, 
respectively.
\end{example}

\begin{definition}
Let $I$ be an ideal in a polynomial ring and $\prec$ a monomial order.
A monomial which does not belong to $\init_\prec (I)$ is called a standard monomial of $I$ 
with respect to $\prec$. We denote the set of standard monomials by $\B_\prec(I)$. 
\end{definition}

The set of standard monomials of $I$ with respect to a monomial 
order is a basis for the quotient of the polynomial ring by the ideal, as a vector space over the field 
(\emph{cf}.~Macaulay's Theorem, \cite[\S2.2.2]{EnHe12}). Let $e\in E_G$ be a choice of an edge and $\prec$ 
a grevlex order with $t_e$ last. In the final part of this work we will use the standard basis 
of the zero dimensional ideal $(I(X_G),t_e)$ to compute the degree and the regularity of $I(X_G)$. 
Because $(I(X_G),t_e)$ zero dimensional, $\B_\prec(I(X_G),t_e)$ is a finite set and, as we show next, 
can be described in a combinatorial way, using $T$-joins of fixed parity. 
\smallskip

For the sake of clarity of notation, recall that $E_G$ is regarded as set of (unordered) pairs 
of vertices and therefore if $e=\set{a,b}$ is an edge and $T\in \PE(V_G)$ is an even subset of vertices, we may refer 
to $T\sd e =  T\sd\set{a,b}\in \PE(V_G)$.

\begin{theorem}\label{thm: standard monomials and T-sets}
Let $e\in E_G$ and $\prec$ a grevlex order on $K[E_G]$ with $t_e$ last. 
Let $\vartheta \colon \B_\prec(I(X_G),t_e) \to \PE(V_G)\times (\ZZ/2)$ be given by 
\mbox{$\t^\alpha \mapsto (\theta(\t^\alpha), \deg(\t^\alpha) + 2\ZZ)$.} 
Then $\vartheta$ is injective and 
\begin{equation}\label{eq: characterizing the image set}
\operatorname{Im} \vartheta = \set{(T,i) : \J_i(G,T)\not = \emptyset \text{ and } \tau_{i+1}(G,T\sd e) = \tau_i(G,T) + 1}.
\end{equation}
\end{theorem}

\begin{proof} Consider the Gr\"obner basis of $(I(X_G),t_e)$ given by $\G\cup \set{t_e}$, 
where $\G$ is the Gr\"obner basis of $I(X_G)$ given in Definition~\ref{def: Groebner basis}.
Let $\t^\alpha$ and $\t^\beta$, belonging to $\B_\prec(I(X_G),t_e)$, be such that 
$\vartheta(\t^\alpha)=\vartheta(\t^\beta)$. Denote 
$$
\renewcommand{\arraystretch}{1.3}
\begin{array}{l}
T=\theta(\t^\alpha)=\theta(\t^\beta), \\ 
i=\deg(\t^\alpha)+2\ZZ = \deg(\t^\beta)+2\ZZ.
\end{array}
$$
Since $\t^\alpha$ and $\t^\beta$ are standard elements 
of $I(X_G)$ with respect to $\prec$ that, additionally, are not divisible by $t_e$, 
both $\t^\alpha$ and $\t^\beta$ are square-free. Moreover by (ii) of 
Theorem~\ref{thm: remainder of a monomial by the groebner basis}, 
$\J(\t^\alpha)$ and $\J(\t^\beta)$ are minimum cardinality elements of $\J_i(G,T)$.
We conclude that 
$$
\deg(\t^\alpha) = |\J(\t^\alpha)| = |\J(\t^\beta)| = \deg(\t^\beta)
$$
Assume, arguing by contradiction, that $\t^\alpha \not = \t^\beta$. Then, as 
$\theta(\t^\alpha) = \theta(\t^\beta)$ and $\deg(\t^\alpha)=\deg(\t^\beta)$, 
the binomial $\t^\alpha-\t^\beta$ is Eulerian and thus one of $\t^\alpha$ or 
$\t^\beta$ is a leading term of $\G \cup \set{t_e}$, which is a contradiction.
\smallskip 

Let us now prove \eqref{eq: characterizing the image set}.
Suppose that $(T,i)=\vartheta(\t^\alpha)$, for some $\t^\alpha \in \B_\prec(I(X_G),t_e)$. 
As $\t^\alpha$ and $\t^\alpha t_e$ are square-free standard monomials of $I(X_G)$ 
and $\theta(\t^\alpha t_e) = T\sd e$, by (ii) of Theorem~\ref{thm: remainder of a monomial by the groebner basis}, 
$$
\tau_i(G,T)=\deg(\t^\alpha) = \deg(\t^\alpha t_e) -1 = \tau_{i+1}(G,T\sd e)-1.
$$

Conversely, suppose $\J_i(G,T)$ is non-empty and 
that $\tau_{i+1}(G,T\sd e)=\tau_{i}(G,T)+1$. 
Let $J$ be a minimum cardinality element of $\J_i(G,T)$ and 
let $\t^\alpha$ be the remainder in a standard expression with respect to $\G$ of the 
monomial given by the product of all edges in $J$. 
By (i) of Theorem~\ref{thm: remainder of a monomial by the groebner basis}, 
$\J(\t^\alpha) \in \J_i(G,T)$ and by (ii) of the same proposition, 
$|\J(\t^\alpha)|=\tau_i(G,T)=|J|$. This proves that $\t^\alpha$ is square-free, so that
if $t_e$ divides $\t^\alpha$, $|\J(\t^\alpha t_e^{-1})| = |\J(\t^\alpha)|-1$. However, as 
$\J(\t^\alpha t_e^{-1})$ is a $(T\sd e)$-join, this would imply that 
$$
\tau_{i+1}(G,T\sd e)\leq \tau_i(G,T)-1
$$
which is not true. We deduce that $\t^\alpha$ is not divisible by $t_e$ and therefore 
belongs to $\B_\prec(I(X_G),t_e)$. By what was said, 
it is clear that $\vartheta(\t^\alpha) = (T,i)$.
\end{proof}

\subsection*{Degree} The next result was first 
proved in \cite[Proposition~2.11]{NeVPVi20} by reducing to $K=\ZZ/3$ and to the vanishing ideal 
of the projective set parameterized by $G$. We can now give an alternative proof, drawing closely on the 
properties of the graph. Below, $b_0(G)$ denotes the number of connected components of $G$.

\begin{prop}\label{prop: degree}
The degree of $K[E_G]/I(X_G)$ is 
$$
\left \{
\begin{array}{l}
2^{|V_G|-b_0(G)}, \text{ if $G$ is non-bipartite}, \\
2^{|V_G|-b_0(G)-1}, \text{ if $G$ is bipartite.} 
\end{array}
\right.
$$
\end{prop}

\begin{proof}
Fix $e\in E_G$ and $\prec$ a grevlex order on $K[E_G]$ with $t_e$ last. 
As $t_e$ is regular on $K[E_G]/I(X_G)$ and its degree is one, the degree of $K[E_G]/I(X_G)$ 
is equal to the dimension of $K[E_G]/(I(X_G),t_e)$ as 
a vector space over $K$, which is then given by the number of elements of $\B_\prec(I(X_G),t_e)$. 
Let $\vartheta$ be the map of Theorem~\ref{thm: standard monomials and T-sets}.
Given $T\in \PE(V_G)$, consider the subset of $\PE(V_G)\times (\ZZ/2)$ given by  
$$
\FF_T = \set{(T,0),(T,1),(T\sd e , 0),(T\sd e ,1)}. 
$$
We claim that $|\operatorname{Im}\vartheta \cap \FF_T|=1$, if $G$ is bipartite, and 
$|\operatorname{Im}\vartheta \cap \FF_T|=2$, if $G$ is non-bipartite.
From this claim we deduce that the degree
of $K[E_G]/(I(X_G),t_e)$ is equal to $\frac{|\PE(V_G)|}{2}$ if $G$ 
is bipartite or $|\PE(V_G)|$, otherwise. To ease notation, let $r=b_0(G)$ and 
denote by $n_1,\dots, n_r$ the cardinalities of the sets of vertices of the connected components. 
Then $|\PE(V_G)|=2^{n_1-1}\cdots 2^{n_r-1}= 2^{|V_G|-r}$ and thus the result follows.
\smallskip 

Let us now prove the claim. For a fixed $T$, at least one of the sets $\J_0(G,T)$, $\J_1(G,T)$ is non-empty. 
Fix $i \in \ZZ/2$ such that $\J_i(G,T)\not = \emptyset$ and $J\in \J_i(G,T)$ of minimum cardinality. 
Then, as $J\sd \set e\in \J_{i+1}(G,T\sd e)$, the set $\J_{i+1}(G,T\sd e)$ is also non-empty. 
Moreover, since  $|J\sd \set e| \leq |J|+1$, 
$$\tau_{i+1}(G,T\sd e )\leq \tau_i(G,T)+1.$$ 
Repeating this argument with $T\sd e$ and $i+1\in \ZZ/2$ and combining the results, 
$$
\tau_{i+1}(G,T\sd e)-1 \leq \tau_i(G,T) \leq \tau_{i+1}(G,T\sd e)+1.
$$
As $\tau_i(G,T)$ and $\tau_{i+1}(G,T\sd e)$ have opposite parities, either 
$$
\renewcommand{\arraystretch}{1.3}
\begin{array}{l}
\tau_i(G,T)=\tau_{i+1}(G,T\sd e)-1 \text{ or } \\
\tau_i(G,T)=\tau_{i+1}(G,T\sd e)+1.
\end{array}
$$
In either case, using Theorem~\ref{thm: standard monomials and T-sets}, 
we deduce that  
$$
|\set{(T,i),(T\sd e,i+1)}\cap \operatorname{Im}\vartheta| = 1.
$$
If $G$ is bipartite then, by Lemma~\ref{lem: when do odd and even T-joins exist},
only one of $\J_0(G,T)$ or $\J_{1}(G,T)$ is non-empty and thus $|\operatorname{Im}\vartheta \cap \FF_T|=1$.
If $G$ is non-bipartite then both of these sets are non-empty
and hence $|\operatorname{Im}\vartheta \cap \FF_T|=2$. The claim is proved. 
\end{proof}

\subsection*{Regularity}
Theorem~\ref{thm: standard monomials and T-sets} will be used to express $\reg I(X_G)$ in a combinatorial way.
Before we do this, and to explain 
the connection with the initial results in this direction contained in \cite{NeVPVi20}, we need 
the following notion.

\begin{definition}\label{def: parity join}
$J\subset E_G$ is called a \emph{parity join} if and only if 
$|J\cap E_C|\leq \frac{|E_C|}{2}$, for every \emph{even} Eulerian subgraph of $G$.
\end{definition}

This definition is related to the notion of \emph{join} (\emph{cf}.~\cite{Fr93}); 
a subset $J\subset E_G$ is called a join if and only if 
$\ts |J\cap E_C|\leq \frac{|E_C|}{2}$,
for every Eulerian subgraph $C\subset G$. A join is always a parity join but not the way around.  
The relation between 
joins and $T$-joins is established in Guan's Theorem (\emph{cf}.~\cite{Gu60});
if $J$ is a minimum cardinality $T$-join then $J$ is a join and, 
vice-versa, if $J$ is a join then it is a minimum cardinality $T$-join, for 
$T$ equal to the set of odd degree vertices of the 
induced subgraph. A similar result holds for parity joins. 

\begin{lemma}\label{lem: parity joins are minimum fixed parity T-joins}
If $T$ is an even subset of vertices and $i\in \ZZ/2$, 
then any element of $\J_i(G,T)$ of minimum cardinality is a parity join. 
Conversely, any parity join, $J$, is a minimum cardinality 
element of $\J_i(G,T)$, where $T$ is the set of odd degree vertices
of the subgraph induced by $J$ and $i=|J|+2\ZZ$.
\end{lemma}
 
\begin{proof}
Let $J$ be a minimum cardinality element of $\J_i(G,T)$ and let $C$ be an even Eulerian subgraph of $G$. Then 
$J\sd E_C$ is a $T$-join with $|J\sd E_C|\equiv_2 |J|$ hence 
$$
\textstyle |J|\leq |J\sd E_C| \iff |J\cap E_C|\leq \frac{|E_C|}{2},
$$
i.e., $J$ is a parity join.
Conversely, let $J\subset E_G$ be a parity join and $T\subset V_G$ be the set of odd degree vertices of the 
subgraph of $G$ induced by $J$. Then $J$ is a $T$-join. 
Set $i=|J|+2\ZZ$ and let $J'\in \J_i(G,T)$. Then $J\sd J'$ defines an even Eulerian subgraph and therefore,
$$
\renewcommand{\arraystretch}{1.3}
\begin{array}{l}
|J\cap (J\sd J')| \leq \frac{|J\sd J'|}{2} \iff \\
|J|-|J\cap J'| \leq \frac{|J| + |J'|}{2} - |J\cap J'| \iff
|J|\leq |J'|.
\end{array}
$$
We deduce that $J$ is a minimum cardinality element of $\J_i(G,T)$.
\end{proof}

\begin{theorem}\label{thm: regularity and parity joins}
The regularity of $I(X_G)$ is the maximum cardinality of minimum 
fixed parity $T$-joins or, equivalently, the maximum cardinality of parity joins.
\end{theorem}

\begin{proof}
Let $e\in E_G$ and fix $\prec$ a grevlex order on $K[E_G]$ with $t_e$ last.
By Proposition~\ref{prop: regularity reduction}, we get
$\reg I(X_G)=\max\set{d: N_d\not = 0}+1$, where $N=K[E_G]/(I(X_G),t_e)$.
Another way of expressing this, using $\B_\prec(I(X_G),t_e)$, is
$$
\reg I(X_G) = \max \set{\deg(\t^\alpha) : \t^\alpha \in \B_\prec(I(X_G),t_e)}+1.
$$
Let $\t^\alpha\in \B_\prec(I(X_G),t_e)$ be of maximum degree and let $\vartheta$ 
be as in Theorem~\ref{thm: standard monomials and T-sets}. Denote  
\mbox{$(T,i)=\vartheta (\t^\alpha)$}. Then, 
by (ii) of Theorem~\ref{thm: remainder of a monomial by the groebner basis} 
and Theorem~\ref{thm: standard monomials and T-sets}, 
$$
\deg(\t^\alpha)+1=\tau_i(G,T)+1= \tau_{i+1}(G,T\sd e)
$$ 
and so $\reg I(X_G) = \tau_{i+1}(G,T\sd e)$.
We conclude that
$$
\reg I(X_G) \leq \max \set{\tau_i(G,T) : T\in \E(V_G),\; i\in \ZZ/2,\; \J_i(G,T)\not = \emptyset}.
$$
To prove the opposite inequality, let $T_0\in\PE(V_G)$, $k\in \ZZ/2$ be such that
$\tau_k(G,T_0)$ is the maximum of the set above. Fix $J\in \J_k(G,T_0)$ with 
$|J|=\tau_k(G,T_0)$ and let $\t^\alpha$ be the remainder in a standard expression 
with respect to $\G$ of the monomial given by the product of all edges in $J$. Then,
arguing as in the proof of Theorem~\ref{thm: standard monomials and T-sets}, 
we deduce that $\t^\alpha$ is square-free,  
$\J(\t^\alpha) \in \J_k(G,T_0)$ and \mbox{$\deg(\t^\alpha) = \tau_k(G,T_0)$}.
By the maximality of $\tau_k(G,T_0)$ we get 
$$
\tau_{k+1}(G,T_0\sd e)\leq \tau_k(G,T_0)
$$ 
which, by Theorem~\ref{thm: standard monomials and T-sets}, 
means that $(T_0,k)\not \in \operatorname{Im} \vartheta$.
This implies that $\t^\alpha$ is not a standard element
of $(I(X_G),t_e)$; hence $t_e$ divides $\t^\alpha$ and $\t^\alpha t_e^{-1} \in \B_\prec(I(X_G),t_e)$.
Accordingly, $\reg I(X_G) \geq \deg(\t^\alpha t_e^{-1}) +1 = \deg(\t^\alpha) = \tau_k(G,T_0)$.
\end{proof}

In \cite[Theorem~4.5]{NeVPVi20} it was shown that for a bipartite graph the
re\-gu\-la\-ri\-ty of $I(X_G)$ is equal to the maximum cardinality of a join. 
Theorem~\ref{thm: regularity and parity joins} is therefore a generalization of this result.

\begin{example}
By Theorem~\ref{thm: regularity and parity joins}, the value of $\reg I(X_G)$, for
$G=\K_n$, can now be obtained by an analysis of the minimum fixed parity $T$-joins. 
(This was done in \cite{NeVPVi20} by reducing to $K=\mathbb{Z}/3$ 
and using the results of \cite{GoReSa13}.)
From Table~\ref{tbl: minimum and odd T-joins in Kn}, if $n=3$, $\reg I(X_G)=3$; obtained
by taking a minimum cardinality element of $\J_1(G,\emptyset)$.
(Note that $G=\K_3$ is listed in Table~\ref{tbl: values for special families of graphs} in the family of 
non-bipartite uni-cyclic graphs.) If $n\geq 4$, the maximum cardinality of a fixed parity $T$-join is 
$r = \lfloor \nicefrac{n}{2}\rfloor +1$. Denoting $i=r+2\ZZ$, we see that $r=\tau_i(G,V_G)$,
if $n$ is even, and, if $n$ is odd, $r=\tau_i(G,T)$ with 
$T=V_G\setminus \set{v}$, for any choice of $v\in V_G$.
\end{example}


\begin{thebibliography}{99}

\bibitem{BiO'KVT17} J.~Biermann, A.~O'Keefe and A.~Van Tuyl,
\emph{Bounds on the regularity of toric ideals of graphs},
Adv.~Appl.~Math.~85 (2017), 84--102.

\bibitem{Ei05} D.~Eisenbud,  
\emph{The geometry of syzygies. A second course in commutative algebra and algebraic geometry},
Graduate Texts in Mathematics, 229, Springer-Verlag, New York, 2005.

\bibitem{EnHe12} V.~Ene and J.~Herzog,
\emph{Gr\"obner bases in commutative algebra},
Graduate Studies in Mathematics, 130. American Mathematical Society, Providence, RI, 2012.

\bibitem{FaKeVT20} G.~Favacchio, G.~Keiper and A.~Van Tuyl,
\emph{Regularity and $h$-polynomials of toric ideals of graphs},
Proc.~Amer.~Math.~Soc.~148 (2020), no. 11, 4665--4677.

\bibitem{Fr93} A.~Frank, 
\emph{Conservative weightings and ear-decompositions of graphs}, 
Combinatorica 13 (1993), no.~1, 65--81. 
 
\bibitem{GeeKa18} J.~Geelen and R.~Kapadia, 
\emph{Computing girth and cogirth in perturbed graphic matroids}, 
Combinatorica 38 (2018), no.~1, 167--191.

\bibitem{GoReSa13} M.~Gonz\'alez, C.~Renter\'ia and E.~Sarmiento,
\emph{Parameterized codes over some embedded sets and their applications to complete graphs},
Math.~Commun., 18 (2013), no.~2, 377--391.

\bibitem{Gu60} M.~Guan,
\emph{Graphic programming using odd or even points},
Acta Math.~Sinica, 10 (1960), 263--266.

\bibitem{HaBeO'K19} H.~T.~H\`a, S.~K.~Beyarslan and A.~O'Keefe, 
\emph{Algebraic properties of toric rings of graphs},  
Comm.~Algebra, 47 (2019), no. 1, 1--16. 

\bibitem{HeHi11} J.~Herzog and T.~Hibi, 
\emph{Monomial ideals}. 
Graduate Texts in Mathematics, 260. Springer-Verlag London, Ltd., London, 2011. 

\bibitem{HeHi20} J.~Herzog and T.~Hibi, 
\emph{The Regularity of Edge Rings and Matching Numbers}. 
Mathematics. (2020) 8, 39.  
\href{https://doi.org/10.3390/math8010039}{doi: 10.3390/math8010039}

\bibitem{KoVy06} B.~Korte and J.~Vygen,
\emph{Combinatorial optimization.}
Theory and algorithms. Third Edition. Algorithms and Combinatorics, 21. 
Springer-Verlag Berlin Heildelberg, 2006. 

\bibitem{NeVPVi20} J.~Neves, M. Vaz~Pinto and R.~H.~Villarreal,
\emph{Joins, Ears and Castelnuovo-Mumford regularity}, 
J.~Algebra, 560 (2020), 67--88.

\bibitem{ReSiVi11} C.~Renter\'\i a, A.~Simis and R.~H.~Villarreal,
\emph{Algebraic methods for parameterized codes and invariants of vanishing ideals over finite fields}, 
Finite Fields Appl., 17 (2011), no. 1, 81--104.

\bibitem{Sch03} A.~Schrijver,
\emph{Combinatorial optimization. Polyhedra and efficiency ---  
Vol. A. Paths, flows, matchings. Chapters 1--38}, 
Algorithms and Combinatorics, 24, A. 
Springer-Verlag, Berlin, 2003. 

\bibitem{SiVaVi94} A.~Simis, V.~Vasconcelos and R.~H.~Villarreal,
\emph{On the ideal theory of graphs},
J.~Algebra 167 (1994), no.~2, 389--416.

\bibitem{SoZa93} P.~Sol\'e and T.~Zaslavsky, 
\emph{The covering radius of the cycle code of a graph}, 
Discrete Appl.~Math.~45 (1993), no.~1, 63--70. 

\bibitem{Vi95} R.~H.~Villarreal,
\emph{Rees algebras of edge ideals},
Comm.~Algebra 23 (1995), no.~9, 3513--3524.

\end{thebibliography}
\end{document}